\documentclass[a4paper,10pt]{amsart}

\usepackage[cp1250]{inputenc}
\usepackage{graphicx}
\usepackage{amsmath,amsthm}
\usepackage{amssymb}
\usepackage{amscd}
\usepackage{eucal}

\newcommand{\esssup}{\operatornamewithlimits{ess\ sup}\limits}
\newcommand{\BB}{\mathbb}

\newcommand{\on}{\operatorname}

\newcommand{\mc}{\mathcal}

\newcommand{\C}{\BB C}

\newcommand{\sbq}{\subseteq}

\newcommand{\N}{\mathbb N}
\newcommand{\R}{\mathbb R}
\newcommand{\ve}{\varepsilon}

\newcommand{\pa}{||}

\newtheorem{thm}{Theorem}

\newtheorem{lem}[thm]{Lemma}

\newtheorem{cor}[thm]{Corollary}

\theoremstyle{definition}

\theoremstyle{remark}

\title{Uniform openness of multiplication in Banach spaces $L _p$}

\author{Marek Balcerzak}
\address{Institute of Mathematics,
         \L\'od\'z University of Technology, ul. W\'olcza\'nska 215,
         93-005 \L\'od\'z,
         Poland}
\email{marek.balcerzak@p.lodz.pl}
\author{Adam Majchrzycki}
\address{Institute of Mathematics,
         \L\'od\'z University of Technology, ul. W\'olcza\'nska 215,
         93-005 \L\'od\'z,
         Poland}
\email{majchrzu@wp.pl}
\author{Filip Strobin}
\address{Institute of Mathematics,
         \L\'od\'z University of Technology, ul. W\'olcza\'nska 215,
         93-005 \L\'od\'z,
         Poland}
\email{filip.strobin@p.lodz.pl}

\date{}

\begin{document}
\begin{abstract}
We show that multiplication from $L_p\times L_q$ to $L_1$ (for $p,q\in [1,\infty]$, $1/p+1/q=1$)
is a uniformly open
mapping. We also prove the uniform openness of the multiplication from $\ell_1\times c_0$ to $\ell_1$.
This strengthens the former results obtained by M. Balcerzak, A.~Majchrzycki and A. Wachowicz.
\end{abstract}

\maketitle
\section{Introduction}
Let $\BB N=\{ 1,2,\dots\}$.
Assume that $X,Y$ are metric spaces. A mapping $f\colon X\to Y$ is called {\em open
at a point} $x_0\in X$ if for every $\ve >0$ there exists $\delta >0$ such that
$\on{B}(f(x_0),\delta)\sbq f[\on{B}(x_0,\ve)]$, where $\on{B}(z,r)$ stands for the open ball
with center $z$ and radius $r>0$ in a given space.
Then $f$ is {\em open} whenever it is open at every point $x\in X$.
In \cite{BMW}, $f$ is called {\em uniformly open} if for every $\ve >0$ there exists $\delta >0$
such that $\on{B}(f(x),\delta)\sbq f[\on{B}(x,\ve)]$ for all $x\in X$. Note that
an open mapping need not be uniformly open since $\arctan$ serves as a simple example (cf. \cite{BMW}).

The classical Banach open principle states that every linear continuous surjection between two Banach spaces is an open mapping.
It is known that the counterpart of this theorem for bilinear continuous surjections is false.
See \cite[Chapter 2, Exercise 11]{Rud} and \cite{H}. Note that multiplication is a natural example of a bilinear continuous
surjection for several Banach spaces. However, it need not be an open mapping, for instance it is not open in the Banach algebra
$C[0,1]$ of real-valued continuous functions on $[0,1]$ endowed with the supremum norm (see \cite{BWW}). On the other hand, if $C[0,1]$ is
replaced by $C_\C [0,1]$,
the respective Banach algebra of complex-valued continuous functions, the multiplication is open \cite{Bh}.
In \cite{BMW} one can find other examples where multiplication in function spaces is open or even uniformly open.

Let $X,Y,Z$ be Banach spaces and suppose that the multiplication $(x,y)\mapsto\Phi(x,y):=xy$ is a well defined operator from
$X\times Y$ to $Z$. Then for $A\sbq X$, $B\sbq Y$, we denote $\Phi[A\times B]$ by $A\cdot B$.
The uniform openness of $\Phi$ means that  for every $\ve >0$ there exists $\delta >0$
such that $\on{B}(xy,\delta)\sbq \on{B}(x,\ve)\cdot\on{B}(x,\ve)$ for all $(x,y)\in X\times Y$.

In \cite[Prop. 1]{BMW} it was shown that the multiplication in $\BB R$ is a uniformly open mapping. Using a similar proof, we extend it slightly as follows.

\begin{lem} \label{L1}
For any $r,R>0$ and $x,y\in\R$ we have $$\left(xy-\dfrac{rR}{4},xy+\dfrac{rR}{4}\right)\subseteq (x-r,x+r)\cdot (y-R,y+R).$$
\end{lem}
\begin{proof}
Fix $r,R>0$ and $x,y\in\R$. Let $|z-xy|<\dfrac{rR}{4}$. Consider three cases:
\begin{itemize}
\item[$1^0$] $|x|>r/4$. Put $u=x$ and $v=z/x$. Then $z=uv$ and $|u-x|<r$. Also
$$|v-y|=\frac{|z-xy|}{|x|}<\frac{rR}{4}\cdot\dfrac{4}{r}=R$$
\item[$2^0$] $|y|>R/4$. Put $v=y$ and $u=z/y$. The rest is analogous to $1^0$.
\item[$3^0$] $|x|\leq{r/4}$ and $|y|\leq{R/4}$. Put $u=\sqrt{\dfrac{|z|r}{R}}$, $v=\sqrt{\dfrac{|z|R}{r}}\on{sgn} z$.
Then $z=uv$ and
\begin{equation*}
\begin{array}{rcl}
|u-x|&\le& |u|+|x|\le\sqrt{\dfrac{|z|r}{R}}+\dfrac{r}{4}\le\sqrt{\dfrac{r}{R}}\sqrt{|z-xy|}+\sqrt{\dfrac{r}{R}}\sqrt{|xy|}+
\dfrac{r}{4}\\&<&\sqrt{\dfrac{r}{R}}\dfrac{\sqrt{rR}}{2}+\sqrt{\dfrac{r}{R}}\dfrac{\sqrt{rR}}{4}+\dfrac{r}{4}=r.
\end{array}
\end{equation*}
Similarly, $|v-y|<R$.
\end{itemize}
\end{proof}

\section{The results}
Fix a measure space $(X,\mc S,\mu)$, where $\mu$ is a measure on the $\sigma$-algebra $\mc S$ of subsets of $X$.
We will consider the respective Banach spaces $L_p=L_p(X)$ with $p\in [1,\infty]$.
For $p,q\in [1,\infty]$, $1/p+1/q=1$, consider the multiplication $\Phi\colon L_p\times L_q\to L_1$
(given by $\Phi(f,g)=fg$). Note that this is a bilinear continuous surjection (the fact that it is well defined follows from the H\"{o}lder inequality).
One of the main results of \cite{BMW} states
that this mapping is open. Our main theorem improves it by showing that the mapping is uniformly open.

For $p\in[1,\infty]$, the norm in $L_p$ will be written as $\pa\cdot\pa_p$. By $\on{B}_p(f,r)$ and $\overline{\on{B}}_p(f,r)$
we denote the respective open and closed balls in this space.

\begin{thm} \label{TT}
Let $p\in[1,\infty)$ and let $q\in (1,\infty]$ be such that $1/p+1/q=1$. Then multiplication $L_p\times L_q\ni (f,g)\mapsto fg\in L_1$ is a uniformly open
mapping.
More exactly, for any $\varepsilon>0$, $(f,g)\in L_p\times L_q$, we have
$$\on{B}_1\left(fg,\frac{\ve^2}{4}\right)\subseteq \on{B}_{p}({f},{\ve})\cdot\on{B}_{q}({g},{\ve}).$$
\end{thm}

Except for Lemma \ref{L1}, two next lemmas will be needed.

\begin{lem}\label{L3}
For every $f\in L_1$ and every $\ve>0$, there exists $A\in\mc S$ such that
$\mu(A)<\infty$, $\sup\{|f(x)|\colon x\in A\}<\infty$ and $\int_{X\setminus A}|f|\;d\mu<\ve$.
\end{lem}
\begin{proof}
For every $n\in\N$, let $A_n=\{x\in X\colon {1}/{n}\leq|f(x)|\leq n\}$.
Then $(A_n)$ is an increasing sequence of sets in $\mc S$ such that $\bigcup_{n\in\N}A_n=\{x\in X\colon 0<|f(x)|<\infty\}$.
Since $f\in L_1$, we have $\mu(\{x\in X\colon |f(x)|=\infty\})=0$ and hence
$$
\pa f\pa_1=
\int_{\bigcup_{n\in\N}A_n}|f|\;d\mu=\lim_{n\rightarrow\infty}\int_{A_n}|f|\;d\mu.
$$
In particular, there exists $k\in\N$ such that
$$
\int_{X\setminus A_k}|f|\;d\mu=\int_{(\bigcup_{n\in\N}A_n)\setminus A_k}|f|\;d\mu<\ve.
$$
Clearly, for every $x\in A_k$, $|f(x)|\leq k$, and since also $|f(x)|\geq {1}/{k}$, then $\mu(A_k)<\infty$.
\end{proof}

We say that an $\mc S$-measurable function $f\colon X\rightarrow \R$ is \emph{countably valued}, if $f$ is of the form $\sum_{n\in\N}a_n\chi_{A_n}$,
for a sequence $(a_n)$ of reals and a sequence $(A_n)$ of sets in $\mc S$ which constitute a partition of $X$. Clearly, if $p\in[1,\infty)$ and
$f=\sum_{n\in\N}a_n\chi_{A_n}\in L_p$ is countably valued, then $a_k\neq 0$ implies $\mu(A_k)<\infty$.

\begin{lem}\label{L2}
Let $p\in[1,\infty)$ and let $q\in (1,\infty]$ be such that $1/p+1/q=1$. Then for every $\ve >0$ and any countably valued functions
$f\in L_p$, $g\in L_q$ and $h\in \on{B}_1(fg,{\varepsilon^2}/{4})$, we have  $h\in \on{B}_p(f,\varepsilon)\on{B}_q(g,\varepsilon)$ if $p>1$,
and $h\in \on{B}_p(f,\varepsilon)\overline{\on{B}_q}(g,\varepsilon)$ if $p=1$.
\end{lem}
\begin{proof}
Let $\ve >0$ and let $f\in L_p$, $g\in L_q$, $h\in L_1$ be countably valued functions such that $\pa fg-h\pa_1<{\ve^2}/{4}$.
Then there exist a sequence $(A_n)$ of sets in $\mc S$ that form a partition of $X$, and sequences of reals $(x_n)$, $(y_n)$, $(z_n)$ such that
$$
f=\sum_{n\in\N}x_n\chi_{A_n},\;\;\;g=\sum_{n\in\N}y_n\chi_{A_n}\;\;\mbox{ and }\;\;h=\sum_{n\in\N}z_n\chi_{A_n}.
$$
(Clearly, we can fix a common partition of $X$ which determines the expressions of $f$, $g$ and $h$; then $\mu(A_n)<\infty$ for all $n\in\BB N$.)

Let $E=\{ n\in\BB N\colon z _n=x_ny_n\mbox{ or }\mu(A_n)=0\}$. For all $n\in E$, put
\begin{itemize}
\item $u_n=x_n$ and $v_n=y_n$ if $z_n=x_ny_n$;
\item $u_n=\sqrt{|z_n|}$ and $v_n=\sqrt{|z_n|}\on{sgn}(z_n)$ if $\mu(A_n)=0$ and $z_n\neq x_ny_n$.
\end{itemize}
If $E=\N$, the assertion is trivial. So, assume that $E\neq\BB N$.
Put $$\eta=\pa fg-h\pa_1=\sum_{n\notin E}|z_n -x_n y_n|\mu(A_n)$$ and for every $n\notin E$, $$\lambda_n ={|z_n -x_n y_n|}\mu(A_n)/{\eta}.$$
Then $$\eta<\ve^2 /4,\;\;\lambda_n\in(0,1]\mbox{ for all }n\notin E,\;\;\mbox{ and }\sum_{n\notin E}\lambda_n =1.$$
The next part of the proof is divided into two cases.

{\bf Case 1: $p>1$.} For every $k\notin E$, we have
$$|z_k -x_k y_k|=\frac{\lambda_k\eta}{\mu(A_k)}<\frac{\lambda_k}{\mu(A_k)}\dfrac{\ve^2}{4}=
\dfrac{1}{4}\left(\ve\left(\frac{\lambda_k}{\mu\left(A_k\right)}\right)^{1/p}\right)
\left(\ve\left(\frac{\lambda_k}{\mu\left(A_k\right)}\right)^{1/q}\right).$$
Applying Lemma \ref{L1}
to $r_k=\ve\left(\frac{\lambda_k}{\mu\left(A_k\right)}\right)^{1/p}$ and $R_k=\ve\left(\frac{\lambda_k}{\mu\left(A_k\right)}\right)^{1/q}$,
we find $u_k, v_k \in\R$ such that $z_k =u_k v_k$ and $|u_k -x_k|<r_k$, $|v_k -y_k|<R_k$.
Set $u=\sum_{n\in\N}u_n\chi_{A_n}$ and $v=\sum_{n\in\N}v_n\chi_{A_n}$. Then $h=uv$ and
$$\pa u-f\pa_p=\left(\sum_{n\in\BB N}|u_n-x_n|^p\mu(A_n)\right)^{1/p}=\left(\sum_{n\notin E}|u_n -x_n|^{p}\mu(A_n)\right)^{1/p}$$
$$<\left(\sum_{n\notin E}\frac{\lambda_n}{\mu(A_n)}\ve^p\mu(A_n)\right)^{1/p}=\ve\left(\sum_{n\notin E} \lambda_n\right)^{1/p}
=\ve ,$$
and we prove $\pa v-g\pa_q<\ve$ in an analogous way.

{\bf Case 2:} $p=1$ (then $q=\infty$). For every $k\notin E$, we have
$$|z_k -x_k y_k|=\frac{\lambda_k\eta}{\mu(A_k)}<\frac{\lambda_k}{\mu(A_k)}\dfrac{\ve^2}{4}=\dfrac{1}{4}\left(\frac{\lambda_k \ve}{\mu\left(A_k\right)}\right)\ve.$$
Applying Lemma \ref{L1}
to $r_k=\frac{\lambda_k \ve}{\mu\left(A_k\right)}$ and $R_k=\ve$, we find $u_k, v_k \in\R$ such that $z_k =u_k v_k$ and $|u_k -x_k|<r_k$, $|v_k -y_k|<R_k$.
As previously, set $u=\sum_{n\in\N}u_n\chi_{A_n}$ and $v=\sum_{n\in\N}v_n\chi_{A_n}$. Then $h=uv$ and
$$\pa u-f\pa_1=\sum_{n\in\BB N}|u_n-x_n|\mu(A_n)=\sum_{n\notin E}|u_n -x_n|\mu(A_n)$$
$$<\sum_{n\notin E}\frac{\lambda_n \ve}{\mu(A_n)}\mu(A_n)=\ve\sum_{n\notin E} \lambda_n=\ve ,$$
and $$\pa v-g\pa_\infty=\sup_{n\notin E}|v_n-y_n|\le\ve . $$
\end{proof}

{\bf Proof of Theorem \ref{TT}.}
Let $\ve >0$, $(f,g)\in L_p\times L_q$ and consider $h \in \on{B}_1(fg,{\varepsilon^2}/{4})$.
We first deal with the case when $\mu(X)<\infty$ and $f,g,h$ are bounded.
Let $M>0$ be such that $\mu(X)<M$ and $\sup_{x\in X}\{|f(x)|,|g(x)|,|h(x)|\}\le M$. Let
 $\ve_1>0$ be such that
\begin{equation}
\ve_1+\sqrt{4\pa h-fg\pa_1+8\ve_1}<\ve.\label{b}
\end{equation}
Since $\sqrt{4\pa h-fg\pa_1}<\ve$, we can choose such an $\ve_1$. Now let $\delta>0$ be such that
\begin{equation}
\delta<\ve_1\min\left\{\frac{1}{M^{1/p}},\frac{1}{M^{1/q}},\frac{1}{2M^2}\right\}\label{111}
\end{equation}
(if $q=\infty$ then $M^{1/q}:=1$).
Let $f'$ be a countably valued function defined by
\begin{itemize}
\item $f'(x)=k\delta$, if $f(x)\in [k\delta,(k+1)\delta)$ and $k\in\N\cup\{0\}$;
\item $f'(x)=-k\delta$, if $f(x)\in [(-k-1)\delta,-k\delta)$ and $k\in\N\cup\{0\}$.
\end{itemize}
Let $g'$ be countably valued function associated with $g$ in an analogous way. Then for all $x\in X$ we have
$$
|f(x)-f'(x)|\leq \delta,\;\;\;|g(x)-g'(x)|\leq \delta,\;\;\;|f'(x)|\leq M,\;\;\;|g'(x)|\leq M
$$
and
$$
|f(x)g(x)-f'(x)g'(x)|\leq |f(x)g(x)-f'(x)g(x)|+|f'(x)g(x)-f'(x)g'(x)|\leq M\delta + M\delta=2M\delta.
$$
Hence, by (\ref{111}), we obtain (recall that $M^{1/\infty}=1$):
\begin{equation}
\pa f-f'\pa_p\leq \delta M^{\frac{1}{p}}<\ve_1\;\mbox{, }\;\pa g-g'\pa_q\leq\delta M^\frac{1}{q}<\ve_1\;\mbox{ and }\pa fg-f'g'\pa_1\leq 2M^2\delta<\ve_1.\label{d}
\end{equation}
Now, let $d>0$ be such that
\begin{equation}
d<1,\label{e}
\end{equation}
\begin{equation}
(1-d)M^2<\ve_1,\label{e5}
\end{equation}
\begin{equation}
\ve_1+\frac{1-d}{d}M^{1+\frac{1}{p}}+\frac{1}{d}\sqrt{4\pa h-fg\pa_1+8\ve_1}<\ve.\label{e2}
\end{equation}
By (\ref{b}), such a choice of $d$ is possible. Let $a_n=d^nM$ for $n\in\BB N\cup\{ 0\}$. Since $d\in(0,1)$, we have $a_n\searrow 0$.

Define $h'$ in the following way:
\begin{itemize}
\item $h'(x)=0$ if $h(x)=0$;
\item $h'(x)=a_{n+1}$ if $h(x)\in[a_{n+1},a_{n})$ and $n\in\BB N\cup\{ 0\}$;
\item $h'(x)=-a_{n+1}$ if $h(x)\in(-a_n,-a_{n+1}]$ and $n\in\BB N\cup\{ 0\}$.
\end{itemize}
Clearly, $h'$ is countably valued, bounded by $M$, and for every $x\in X$ with $h(x)\neq 0$,
\begin{equation}1\leq\frac{h(x)}{h'(x)}\leq \frac{a_n}{a_{n+1}}=\frac{1}{d}.\label{bb}\end{equation}
Moreover by (\ref{e}), $|h(x)-h'(x)|\leq M-dM=(1-d)M$ for every $x\in X$, hence by (\ref{d}) and (\ref{e5}) we have
\begin{equation}
\pa h'-f'g'\pa_1\leq\pa h'-h\pa_1+\pa h-fg\pa_1+\pa fg-f'g'\pa_1< M^2(1-d)+\pa h-fg\pa_1+\ve_1\leq \pa h-fg\pa_1+2\ve_1.\label{e1}
\end{equation}
Let
\begin{equation}
\overline{\ve}=\sqrt{4\pa h-fg\pa_1+8\ve_1}.\label{f}
\end{equation}
Then by (\ref{e1}), $\pa h'-f'g'\pa_1<{\overline{\ve}^2}/{4}$ and
by Lemma \ref{L2}, there are functions $u\in L_p$ and $v\in L_q$ such that $h'=uv$, $u\in B_p(f',\overline{\varepsilon})$
and $v\in B_q(g',\overline{\varepsilon})$ if $p>1$, and  $v\in \overline{B_q}(g',\overline{\varepsilon})$ if $p=1$.

Now, define a function $\alpha\colon X\to \R$ in the following way:
\begin{itemize}
\item $\alpha(x)=1$ if $h'(x)=0$;
\item $\alpha(x)=\frac{h(x)}{h'(x)}$  if $h'(x)\neq 0$;
\end{itemize}
By (\ref{bb}), $1\leq\alpha(x)\leq {1}/{d}$ for every $x\in X$ and $h=\alpha h'=(\alpha u)v$. Finally, by (\ref{b}), (\ref{d}), (\ref{e2}), (\ref{f}) we have
$$
\pa g-v\pa_q\leq \pa g-g'\pa_q+\pa g'-v\pa_q<\ve_1+\overline{\ve}<\ve
$$
and
$$
\pa f-\alpha u\pa_p\leq \pa f-f'\pa_p+\pa f'-\alpha u\pa_p\leq\pa f-f'\pa_p+\pa f'-\alpha f'\pa_p+\pa \alpha f'-\alpha u\pa_p
$$
$$
\leq\ve_1+\pa(\alpha-1)f'\pa_p+\pa\alpha(f'-u)\pa_p\leq \ve_1+\left(\frac{1}{d}-1\right)\pa f'\pa_p+\frac{1}{d}\pa f'-u\pa_p\leq
$$
$$
\leq \ve_1+\frac{1-d}{d}M^{1+1/p}+\frac{1}{d}\overline{\ve}<\ve.
$$
This ends the case when $\mu(X)<\infty$ and $f,g$ and $h$ are bounded.
Note that, if $p=1$, $q=\infty$, the above reasoning works without the boundedness of $g$. Indeed, $g\in L_\infty$ means that $g$ is essentially bounded
and it suffices to choose $M>0$ so that $\pa g\pa_\infty\le M$ and $\sup_{x\in X}\{|f(x)|,|h(x)|\}\le M$.
This will be used below in Case 2.

Now, we deal with a general case (where $\ve,f,g,h$ have the meanings as before).
Let $\delta\in (0,\ve)$ be such that
\begin{equation}
\pa h-fg\pa_1<\frac{\delta^2}{4},\label{h1}
\end{equation}
and let $\gamma>0$ be such that
\begin{equation}
\delta+2\gamma<\ve.\label{h2}
\end{equation}
Then we consider two cases.

{\bf Case 1:  $p>1$.}
Using Lemma \ref{L3} for the function $x\mapsto\max\{|f(x)|^p,|g(x)|^q,|h(x)|\}$, we obtain a set $A\in\mc S$ such that $\mu(A)<\infty$,
$f\vert_A$, $g\vert_A$, $h\vert_A$ are bounded and
$$
\left(\int_{X\setminus A}|f|^p\;d\mu\right)^{\frac{1}{p}}<\gamma,\;\;\left(\int_{X\setminus A}|g|^q\;d\mu\right)^{\frac{1}{q}}<\gamma\;\;
\left(\int_{X\setminus A}|h|\;d\mu\right)^{\frac{1}{p}}<\gamma\;\mbox{ and }\left(\int_{X\setminus A}|h|\;d\mu\right)^{\frac{1}{q}}<\gamma.
$$
Using (\ref{h1}) and the first part of the proof for the space $(A,\Sigma|_A,\mu|_{A})$,
we infer that there exist $u\in L_p(A)$, $v\in L_q(A)$ such that $h(x)=u(x)v(x)$ for $x\in A$ and
$$
\left(\int_A|f-u|^p\;d\mu\right)^\frac{1}{p}<\delta\;\;\mbox{and}\;\;\left(\int_A|g-v|^q\;d\mu\right)^\frac{1}{q}<\delta.
$$
Additionally, define $u(x)=|h(x)|^\frac{1}{p}$ and $v(x)=|h(x)|^\frac{1}{q}\on{sgn}(h(x))$ for $x\notin A$. Then $h=uv$ and by (\ref{h2}), we have
$$
\pa f-u\pa_p\leq \left(\int_{A}|u-f|^p\;d\mu\right)^\frac{1}{p}+
\left(\int_{X\setminus A}\left(|h|^{\frac{1}{p}}\right)^p\;d\mu\right)^\frac{1}{p}+\left(\int_{X\setminus A}|f|^p\;d\mu\right)^\frac{1}{p}<\delta+2\gamma< \ve
$$
and analogously, $\pa g-v\pa_q<\ve$.

{\bf Case 2:} $p=1$ (then $q=\infty$).
Using Lemma \ref{L3} for the function $x\mapsto\max\{|f(x)|,|h(x)|\}$, we obtain a set $A\in\mc S$ such that $\mu(A)<\infty$, $f\vert_A$,
$h\vert_A$ are bounded and
$$
\int_{X\setminus A}|f|\;d\mu<\gamma,\;\;\int_{X\setminus A}|h|\;d\mu <\gamma^{2}.
$$
Using (\ref{h1}) and the first part of the proof for the space $(A,\Sigma\vert_A,\mu\vert_{A})$,
we infer that there exist $u\in L_1(A)$, $v\in L_\infty(A)$ such that $h(x)=u(x)v(x)$ for $x\in A$ and
$$
\int_A|f-u|\;d\mu<\delta\;\;\mbox{and}\;\;\esssup_{x\in A}|(g-v)(x)|\le\delta.
$$
Denote $G=\esssup_{x\in X\setminus A}|g(x)|$.
If $x\in X\setminus A$, we additionally define
\begin{itemize}
\item $v(x)=(k+1)\gamma$ if $g(x)\in [k\gamma,(k+1)\gamma)\cap [0,G]$ and $k\in \N\cup\{ 0\}$;
\item $v(x)=-(k+1)\gamma$ if $g(x)\in [(-k-1)\gamma,-k\gamma)\cap [-G,0)$ and $k\in \N\cup\{ 0\}$;
\item $v(x)=1$ otherwise (this holds on a set of measure zero).
\end{itemize}
Let also $u(x)=h(x)/v(x)$ for $x\notin A$. Then $h=uv$ and by (\ref{h2}) we have
$$
\pa g-v\pa_\infty=\esssup_{x\in X}|(g-v)(x)|\le\max\{\delta,\gamma\}<\varepsilon,
$$
$$
\pa f-u\pa_1\leq \int_{A}|f-u|\;d\mu+\int_{X\setminus A}|f-u|\;d\mu<\delta+\int_{X\setminus A}|f|\;d\mu+\int_{X\setminus A}|u|\;d\mu
$$
$$
<\delta +\gamma +\int_{X\setminus A}\left\vert\frac{h}{v}\right\vert\;d\mu
\leq\delta +\gamma +\frac{1}{\gamma}\int_{X\setminus A}|{h}|\;d\mu<\delta+2\gamma<\varepsilon.
$$
\hfill $\Box$

Now, consider a particular case of $(X,\mc S, \mu)$ where $X=\BB N$, $\mc S=\mc P(\BB N)$ and $\mu$ is a counting measure on $\mc S$.
Then $L_p$ is reduced to $\ell_p$, for each $p\in [1,\infty]$. If we use Lemma \ref{L2} in this case, we obtain the following corollary.

\begin{cor} \label{C1}
Let $p,q\in [1,\infty]$, $1/p+1/q=1$. Then
multiplication from $\ell_p\times\ell_q$ to $\ell_1$ is a uniformly open mapping.
\end{cor}

Instead of multiplication from $\ell_1\times\ell_\infty$ to $\ell_1$,
one can consider multiplication from $\ell_1\times c_0$ to $\ell_1$ which is a continuous open surjection (cf. \cite{BMW}).
We will show its uniform openness which improves the result of \cite{BMW}.

The following lemma is an elementary exercise. See \cite[Exercise 12, Chapter 2]{Ru}.
\begin{lem} \label{LL}
Assume that $\sum_{n\in\N}a_n<\infty$ where $a_n\ge 0$ for all $n\in\BB N$ and $a_n>0$ for infinitely many $n$'s.
If $w_k=(\sum_{n\ge k}a_k)^{1/2}$, $k\in\N$, then $\sum_{n\in\N}(a_n/w_n)\le 2\sum_{n\in\N}(w_{n}-w_{n+1})=2w_1$.
\end{lem}

\begin{thm} \label{T2}
Multiplication from $\ell_1\times c_0$ to $\ell_1$  is a uniformly open mapping. More exactly,
for any $\ve>0$ and $x\in\ell_1$, $y\in c_0$, we have
$$\on{B}_1\left(xy,\frac{\ve^2}{16}\right)\subseteq \on{B}_1({x},{\ve})\cdot (\on{B}_{\infty}({y},{\ve})\cap c_0).$$
\end{thm}
\begin{proof}
Let $\ve>0$ and $x=(x_n)\in\ell_1$, $y=(y_n)\in c_0$. Consider $z\in \on{B}_1({xy},{{\ve^2}/{16}})$.
We will show that $z=uv$ for some $u=(u_n)\in \on{B}_1({x},{\ve})$ and $v=(v_n)\in \on{B}_\infty({y},{\ve})\cap c_0$.
Let $E=\{ n\in\BB N\colon z _n=x_ny_n\}$. For all $n\in E$ put $u_n=x_n$ and $v_n=y_n$.
Then consider two cases.

{\bf Case 1:} $\BB N\setminus E$ is finite. This is similar to Case 2 in Lemma \ref{L2}.

If $E=\N$, the assertion is trivial. So, assume that $E\neq\BB N$.
Put $\eta=\sum_{n\notin E}|z_n -x_n y_n|$ and $\lambda_n ={|z_n -x_n y_n|}/{\eta}$ for all $n\notin E$.
Then
$$\eta<\ve^2 /16,\;\;\lambda_n\in(0,1]\text{ for all }n\notin E,\;\;\text{ and }\sum_{n\notin E}\lambda_n =1.$$
For every $k\notin E$ we have
$$|z_k -x_k y_k|=\lambda_k\eta<\lambda_k\dfrac{\ve^2}{16}=\frac{1}{4}\left(\frac{\lambda_k\ve}{2}\cdot\frac{\ve}{2}\right) .$$
Applying Lemma \ref{L1}
to $r_k=\lambda_k\ve/2$ and $R=\ve/2$, we can find $u_k, v_k \in\R$ such that $z_k =u_k v_k$ and $|u_k-x_k|<r_k$, $|v_k -y_k|<R$. Then
$$\sum_{n\in\BB N}|u_n-x_n|=\sum_{n\notin E}|u_n -x_n|<\frac{\ve}{2}\sum_{n\notin E} \lambda_n<\ve ,$$
$$\sup_{n\in\BB N}|v_n-y_n|=\sup_{n\notin E}|v_n-y_n|\le\frac{\ve}{2}<\ve .$$
Clearly, $(v_n)\in c_0$ since $v_n=y_n$ for all but finitely many $n$'s.

{\bf Case 2:} $\BB N\setminus E$ is infinite.
Let
$$\eta=2\left(\sum_{n\in\BB N}|z_n-x_ny_n|\right)^{1/2}\text{ and }\lambda_k=\frac{|z_k-x_ky_k|}{\eta\left(\sum_{n\ge k}|z_n-x_ny_n|\right)^{1/2}}
\text{ for }k\notin E .$$
From Lemma \ref{LL} it follows that $\sum_{k\notin E}\lambda_k\le 1$.
By the choice of $z$ we have $\eta <\ve/2$. Hence
for every $k\notin E$ we obtain
$$|z_k -x_k y_k|=\lambda_k\eta\left(\sum_{n\ge k}|z_n-x_ny_n|\right)^{1/2}<\frac{1}{4}(\lambda_k\ve)\left(2\left(\sum_{n\ge k}|z_n-x_ny_n|\right)^{1/2}\right) .$$
Applying Lemma \ref{L1}
to $r_k=\lambda_k\ve$ and $R_k=2\left(\sum_{n\ge k}|z_n-x_ny_n|\right)^{1/2}$, we find $u_k, v_k \in\R$ such that
$z_k =u_k v_k$ and $|u_k-x_k|<r_k$, $|v_k -y_k|<R_k$. Then
$$\sum_{n\in\BB N}|u_n-x_n|=\sum_{n\notin E}|u_n -x_n|<\ve\sum_{n\notin E} \lambda_n\le\ve ,$$
$$\sup_{n\in\BB N}|v_n-y_n|=\sup_{n\notin E}|v_n-y_n|\le 2\left(\sum_{n\in\BB N}|z_n-x_ny_n|\right)^{1/2}=\eta<\frac{\ve}{2}<\ve .$$
Since $R_n\to 0$, we have $(v_n-y_n)\in c_0$, hence also $(v_n)\in c_0$. Therefore $u=(u_n)$ and $v=(v_n)$ are as desired.
\end{proof}

\noindent
{\bf Acknowledgement.} We would like to thank Aleksander Maliszewski for his useful comments which simplified a preliminary part of our former reasoning.

\end{document}